\documentclass[12pt,a4paper]{article}

\usepackage{amssymb}
\usepackage{amsmath, amsthm}
\usepackage{arydshln}
\usepackage{epsfig}
\usepackage{setspace}

\usepackage{geometry}

\numberwithin{equation}{section}

\newcommand{\rank}{\mathop{\rm rank} }

\newtheorem{Thm}{Theorem}[section]
\newtheorem{Prop}[Thm]{Proposition}
\newtheorem{Lem}[Thm]{Lemma}

\theoremstyle{definition}

\newtheorem{Rem}[Thm]{Remark}
\newtheorem{Ex}[Thm]{Example}

\title{Computing DM-decomposition of \\a partitioned matrix with rank-1 blocks}
\author{Hiroshi HIRAI \\
Department of Mathematical Informatics, \\
Graduate School of Information Science and Technology,   \\
University of Tokyo, Tokyo, 113-8656, Japan.\\
\texttt{\normalsize hirai@mist.i.u-tokyo.ac.jp}}
\begin{document}
\maketitle
\begin{abstract}
In this paper, we develop a polynomial time algorithm to 
compute a Dulmage-Mendelsohn-type decomposition of a 
matrix partitioned into submatrices of rank at most $1$.
\end{abstract}


MSC classes: 15A21,  15B99, 68Q25.

\section{Introduction}
The {\em Dulmage-Mendelsohn decomposition (DM-decomposition)}~\cite{DulmageMendelsohn58,DulmageMendelsohn59} of an $n \times m$ matrix $A$ is
a canonical block-triangulation under transformation
\[
A \mapsto P^{\top}AQ
\]
for permutation matrices $P$ and $Q$; 
see also \cite[Section 4.3]{LovaszPlummer} and \cite[Section 2.2.3]{MurotaBook}.
The DM-decomposition exploits 
zero submatrices of the largest size ($\geq \max \{n,m\}$), 
where the size 
is defined as the sum of numbers of rows and columns. 
Finding such a zero submatrix of $A$ 
is nothing but the maximum stable set problem in 
the bipartite graph associated with the nonzero pattern of entries of~$A$. 
The family of maximum stable sets is regarded as 
the minimizer set of a submodular function, and  forms a distributive lattice.
The DM-decomposition is obtained by 
arranging rows and columns with respect to
a maximal chain of this distributive lattice. 

The present paper addresses 
a generalization of the DM-decomposition considered
by Ito, Iwata, and Murota~\cite{ItoIwataMurota94}; see also \cite[Section 4.8]{MurotaBook}.
This generalization deals with a matrix $A$ partitioned into submatrices in the following form:
\[
A = \left(
\begin{array}{ccccc}
A_{11} & A_{12} &\cdots & A_{1\nu} \\
A_{21} & A_{22} &\cdots & A_{2\nu} \\
\vdots & \vdots & \ddots & \vdots \\
A_{\mu1}&A_{\mu2} &\cdots & A_{\mu \nu}
\end{array}\right),
\]
where $A_{\alpha \beta}$ is an $n_\alpha \times m_\beta$ matrix for $\alpha=1,2,\ldots,\mu$, 
$\beta=1,2,\ldots,\nu$. 
Such a matrix is called a {\em partitioned matrix of 
type $(n_1,n_2, \ldots, n_\mu; m_1,m_2,\ldots,m_\nu)$}. 
In this setting, 
an admissible transformation is $A \mapsto$
\[
P^{\top} \left(
\begin{array}{cccc}
E_1^{\top} & O  & \cdots & O  \\
O & E_2^{\top} & \ddots & \vdots \\
\vdots & \ddots & \ddots &   O\\
O & \cdots & O& E_{\mu}^{\top}
\end{array}
\right)
\left(
\begin{array}{ccccc}
A_{11} & A_{12} &\cdots & A_{1\nu} \\
A_{21} & A_{22} &\cdots & A_{2\nu} \\
\vdots & \vdots & \ddots & \vdots \\
A_{\mu1}&A_{\mu2} &\cdots & A_{\mu \nu}
\end{array}\right)
\left(
\begin{array}{cccc}
F_1 & O  & \cdots & O  \\
O & F_2 & \ddots & \vdots \\
\vdots & \ddots & \ddots &   O\\
O & \cdots & O& F_{\nu}
\end{array}
\right)
Q,
\]
where $E_\alpha$ is a nonsingular $n_\alpha \times n_\alpha$ matrix for $\alpha = 1,2,\ldots,\mu$, 
$F_\beta$ is a nonsingular $m_\beta \times m_\beta$ matrix 
for $\beta = 1,2,\ldots, \nu$, 
$P$ and $Q$ are permutation matrices of size $n$ and $m$, respectively.
Ito, Iwata, and Murota~\cite{ItoIwataMurota94} showed the existence of
a canonical block-triangulation under this transformation, 
which we call the {\em DM-decomposition}.
This generalization of DM-decomposition is obtained from 
the minimizer set of a submodular function on a modular lattice.

Submodular optimization on modular lattices is 
an undeveloped area of combinatorial optimization, 
and has just been started~\cite{FKMTT14,Kuivinen11}.
It is an open problem in \cite[p. 1252]{ItoIwataMurota94} to design a polynomial time algorithm to compute DM-decomposition of  partitioned matrices.
Currently such an algorithm
is known for very restricted classes of partitioned matrices.  
For a partitioned matrix of type $(n;m)$ with $\mu = \nu = 1$, 
the DM-decomposition is the rank normal form, 
and is computed via Gaussian elimination.
For a partitioned matrix of type 
$(1,1,\ldots,1;1,1,\ldots,1)$, 
the DM-decomposition coincides with the original one~\cite{DulmageMendelsohn58,DulmageMendelsohn59}, 
and is computed via bipartite matching algorithm.
For a partitioned matrix of type 
$(n_1,n_2,\ldots,n_{\mu}; 1,1,\ldots,1)$ 
(or $(1,1,\ldots,1;m_1,m_2,\ldots,m_{\nu})$), 
the DM-decomposition  
is the 
{\em combinatorial canonical form (CCF) of multilayered mixed matrix}~\cite{MurotaBook,MurotaIriNakamura87}, 
and is computed via matroid union algorithm.
To the best of the author's knowledge, 
the computational complexity of other cases is completely unknown.

The main result of this paper is a polynomial time algorithm 
of the DM-decomposition for a new class of partitioned matrices, 
which generalizes the above CCF case.
\begin{Thm}\label{thm:main} 
	Let $A = (A_{\alpha \beta})$ be a partitioned matrix.
	Suppose that each submatrix $A_{\alpha \beta}$ has rank at most $1$.
	Then the DM-decomposition of $A$ is computed in polynomial time.
\end{Thm}

The rest of the paper is organized as follows.
In Section~\ref{sec:pre}, we provide a necessary background 
for lattice and matroid.
In Section~\ref{sec:DM}, we formally introduce the DM-decomposition of partitioned matrices.
Our view is different from Ito, Iwata, and Murota~\cite{ItoIwataMurota94}.
We view a partitioned matrix $A$ as a bilinear form 
on vector space $U \times V$, 
and formulate the problem, 
called {\em maximum stable subspace problem (MSSP)}, of finding 
a vector subspace $(X,Y)$ 
in a specified sublattice of all vector subspaces of $U \times V$
such that $A$ vanishes on $X \times Y$ 
and $\dim X + \dim Y$ is maximum.
This vector-space generalization of 
the bipartite stable set problem 
seems interesting in its own right.
Then the DM-decomposition is obtained by 
a maximal chain of the family of maximum stable subspaces.   
In Section~\ref{sec:DM-rank1}, 
we deal with the special case where each submatrix 
has rank at most $1$.
We show that (MSSP) reduces to 
the {\em maximum independent matching problem}~\cite{Brualdi70}, 
which is a version of matroid intersection problem.
From a maximum independent matching, 
the transformation matrices for the DM-decomposition 
are obtained in polynomial time.

\paragraph{Remark.} After the submission, 
we recognized that significant developments closely related to (MSSP) occured recently; see \cite{HamadaHirai17} and references therein.
We also found that the paper~\cite{Lovasz89} by Lov\'asz in 1989 contains  
an essentially equivalent problem of (MSSP) for a slightly more general setting. 
In this paper, he has already showed that the rank-1 case of (MSSP) reduces to matroid intersection, though the construction of a maximal chain of maximum stable subspaces, which we need,  is not clear.

\section{Preliminaries}\label{sec:pre}

\subsection{Lattice}
A {\em lattice} is a partially ordered set (poset) $({\cal L},\preceq)$
such that every pair $p,q$ of elements 
has meet $p \wedge q$ (greatest common lower bound) 
and join $p \vee q$ (lowest common upper bound).
By $p \prec q$ we mean $p \preceq q$ and $p \neq q$.
A sequence $p_0 \prec p_1 \prec \cdots \prec p_k$ of 
pairwise comparable elements is called a {\em chain} from $p_0$ to $p_k$ where $k$ is called the length. 
In this paper, we only consider lattices with the properties 
that both the minimum element and the maximum element exist and the length of any maximal chain from the minimum to the maximum is finite.
The rank $r(p)$ of element $p$ is defined 
as the maximum length of a chain from the minimum to $p$. 

A lattice ${\cal L}$ is said to be {\em distributive} 
if $x \vee (y \wedge z)= (x \vee y)\wedge (x \vee z)$ 
and $x \wedge (y \vee z)= (x \wedge y)\vee (x \wedge z)$ 
hold for every triple $x,y,z$ of elements.
A canonical example of a distributive lattice 
is the family ${\cal J}({\cal P})$ of all ideals of a (finite) poset ${\cal P}$,  
where an {\em ideal} of ${\cal P}$ is a subset $J \subseteq {\cal P}$ of elements 
with the property that $p \preceq q \in J$ implies $p \in J$.
The partial order of ideal family ${\cal J}({\cal P})$ is given by the inclusion order. 
The following is a simpler version of the Birkhoff representation theorem.
\begin{Lem}
	A lattice ${\cal L}$ is distributive if  and only if 
	it is isomorphic to ${\cal J}({\cal P})$ for some poset~${\cal P}$.
\end{Lem}
A lattice ${\cal L}$ is called {\em modular} if 
for every triple $x,a,b$ of elements with $x \preceq b$, 
it holds $x \wedge (a \vee b) = (x \vee a) \wedge b$.
It is known that a modular lattice is exactly such 
a lattice that satisfies
\[
r(p) + r(q) = r(p \wedge q) + r(p \vee q) \quad (p,q \in {\cal L}).
\]
A canonical example of a modular lattice is 
the family ${\cal U}$ of all subspaces of a vector space $U$, 
where the partial order is the inclusion order.
For two subspaces $X,Y$, 
the meet $X \wedge Y$ is equal to the intersection $X \cap Y$, and
the join $X \vee Y$ is equal to the sum $X+Y$.
The rank of $X$ is equal to the dimension $\dim X$.
The following equality of dimension is well-known:
\begin{equation}\label{eqn:dim}
\dim X+ \dim Y = \dim (X \cap Y) + \dim (X+Y) \quad (X,Y \in {\cal U}).
\end{equation}

\subsection{Matroid}\label{subsec:matroid}
A {\em matroid} ${\bf M} = (V,{\cal I})$ is 
a pair of a finite set $V$
and a family ${\cal I}$ of subsets of $V$ 
such that $\emptyset  \in {\cal I}$, $J \subseteq I \in {\cal I}$ 
implies $J \in {\cal I}$,  and for $I,I' \in {\cal I}$ with $|I| < |I'|$
there is $e \in I' \setminus I$ such that $I \cup \{e\} \in {\cal I}$. 
Here $V$ is called the {\em ground set}, and a member of ${\cal I}$ is called 
an {\em independent set}. 
The {\em rank function} $\rho$ is defined by 
$\rho(X) := \max\{ |I| \mid I \in {\cal I}: I \subseteq X\}$.
The {\em closure operator} ${\rm cl}$ is defined by $
{\rm cl}(X) = \{e \in V \mid \rho(\{e\} \cup X) = \rho(X) \}$.
The {\em direct sum} of two matroids ${\bf M} = (V,{\cal I})$ and 
${\bf M}' = (V', {\cal I}')$ is the matroid such that 
the ground set is the disjoint union $V \cup V'$ and the independent sets 
are $I \cup I'$ for $I \in {\cal I}$ and $I' \in {\cal I}'$.

Let $U$ be a finite-dimensional vector space.
For a finite set $V$ of vectors in $U$ and the family ${\cal I}$ of all linearly independent subsets of $V$, the pair $(V, {\cal I})$ is a matroid. 
We will use a dual construction.
Let $\varPi$ be a finite subset of hyperplanes ($(\dim U-1)$-dimensional subspaces) of $U$
and let ${\cal I}$ be the set of all subsets $H \subseteq \varPi$ 
such that $|H|$ is equal to $\dim U$ minus the dimension of the intersection of hyperplanes in $H$. 
Then $(\varPi, {\cal I})$ is a matroid to be denoted by ${\bf M}(\varPi)$. 
\paragraph{Independent matching problem.}
Following \cite[Section 2.3.5]{MurotaBook}, 
we introduce the independent matching problem.
Let $G = (V^{+}, V^-,E)$ be a bipartite graph 
on vertex set $V^+ \cup V^-$ $(V^+ \cap V^- = \emptyset)$
and edge set $E \subseteq V^+ \times V^-$.
Let ${\bf M}^+$ be a matroid on ground set $V^+$, and ${\bf M}^-$ be a matroid on ground set $V^-$.
The rank functions of ${\bf M}^+$ and ${\bf M}^-$ are denoted by $\rho^+$ and $\rho^-$, respectively.
The closure operators of ${\bf M}^+$ and ${\bf M}^-$ are denoted by
 ${\rm cl}^+$ and ${\rm cl}^-$, respectively.
For an edge subset $F$, let $\partial^+ F$ denote the set of vertices in $V^+$ incident to $F$, and let $\partial^- F$
denote the set of vertices in $V^-$ incident to $F$.
A {\em matching} is an edge subset $M$ with $|M| = |\partial^+ M| = | \partial^- M|$.
A matching $M$ is said to be {\em independent} if 
$\partial^+ M$ is independent in ${\bf M}^+$ 
and $\partial^- M$ is independent in ${\bf M}^-$.
The {\em independent matching problem} 
on $({\bf M}^+, {\bf M}^-,G)$ is the problem of finding 
a matching $M$ of the maximum size. 
The following min-max is a reformulation of Edmonds' matroid intersection theorem, which was obtained by Brualdi~\cite{Brualdi70}:
\begin{Thm}[{\cite{Brualdi70}; see \cite[Theorem 2.3.27]{MurotaBook}}]\label{thm:indep}
	The maximum size of an independent matching is equal to the minimum of
	$
	\rho^+(H) + \rho^-(K)
	$
	over all covers $(H,K)$.
\end{Thm}
Here a {\em cover} $(H,K)$ is a pair of $H \subseteq V^+$ and $K \subseteq V^-$ 
such that every edge meets $H \cup K$.
A cover $(H,K)$ attaining the minimum of $\rho^+(H)+ \rho^-(K)$ is said to be {\em minimum}.

We will use an algorithm to solve 
the independent matching problem and related structures. 
For an independent matching $M$, the {\em auxiliary digraph} $\tilde G_{M}$ is obtained from $G$ as follows.
Orient each edge in $G$ from $V^+$ to $V^-$.
For each $(\pi,\sigma) \in M$, 
add a directed edge from $\sigma$ to $\pi$.
For $\pi' \in \partial^+ M$ and 
$\pi'' \in  {\rm cl}^+(\partial^+ M) \setminus \partial^+ M$,
add a directed edge from $\pi'$ to $\pi''$ 
if $\partial^+ M \cup \{\pi''\} \setminus \{\pi'\}$ 
is independent in ${\bf M}^+$.
For $\sigma' \in {\rm cl}^-(\partial^- M) \setminus \partial^- M$ and 
$\sigma'' \in \partial^- M$, add a directed edge from $\sigma'$ to $\sigma''$  
if $\partial^- M \cup \{\sigma'\} \setminus \{\sigma''\}$
is independent in ${\bf M}^-$.
Let $S := V^+ \setminus {\rm cl}^+ (\partial^+ M)$ and
$T:= V^- \setminus {\rm cl}^- (\partial^- M)$.
\begin{Lem}[{see \cite[Lemma 2.3.32, Theorem 2.3.33, Lemma 2.3.35]{MurotaBook}}]\label{lem:indep}
	\begin{itemize}
	\item[{\rm (1)}] 
	An independent matching $M$ is maximum if and only if there is no directed path 
	in $\tilde G_M$ from $S$ to $T$.
	\item[{\rm (2)}] A cover $(H,K)$ is minimum if and only if 
	it is represented as $(H,K) = (V^+ \setminus C, V^- \cap C)$
	for a vertex subset $C$
	such  that 
	$S\subseteq C$, $T \cap C = \emptyset$, and
	no edge in $\tilde G_{M}$ goes out from $C$.
	\item[{\rm (3)}] For a maximum independent matching $M$ and 
	a minimum cover $(H,K)$, it holds that $\rho^+(H) = |H \cap \partial^+ M|$ and $\rho^-(K) = |K \cap \partial^- M|$.
\end{itemize}
\end{Lem}
The following algorithm to find a maximum independent matching
is due to Tomizawa and Iri~\cite{TomizawaIri74}; see \cite[p. 89]{MurotaBook}.
\begin{description}
	\item[Algorithm to find a maximum independent matching:] 
	\item[0.] $M := \emptyset$.
	\item[1.] Construct $\tilde G_{M}$. 
	Find a shortest path $P$ from $S$ to $T$ in $\tilde G_{M}$.
	\item[2.] If such a path $P$ does not exist, 
	then $M$ is a maximum independent matching; stop.
	\item[3.] Let $E_{P}$ be the set of edges $(\pi,\sigma)$ in $E$ such that
	$(\pi,\sigma)$ or $(\sigma,\pi)$ belongs to $P$. 
	Let $M := (M \setminus E_P) \cup (E_P \setminus M)$, and go to 1.   
\end{description}
Since $P$ is shortest, $M$ is always an independent matching~\cite[Lemma 2.3.31]{MurotaBook}.
The size of $M$ increases by one in each iteration.
Thus the algorithm terminates after at most $|E|$ iterations.
The algorithm needs some matroid oracle to construct $\tilde G_M$.
In our case,  
${\bf M}^+$ and ${\bf M}^-$
are matroids of linear independence of vectors.
Then ${\rm cl}^+(\partial^+ M)$ and ${\rm cl}^-(\partial^- M)$ as well as
the edges to be added for $\tilde G_M$
are calculated by Gaussian elimination.

%

\section{Maximum stable subspace problem and\\ DM-decomposition}
\label{sec:DM}
In this section, we introduce DM-decomposition of partitioned matrix.
Our approach takes a form different from the original approach by Ito, Iwata, and Murota~\cite{ItoIwataMurota94}.
It turns out (in Remark~\ref{rem:surplus}) 
that both approaches yields 
the same definition.
Also we do not impose any rank condition 
on our block-triangular form, 
which simplifies the definition of 
the DM-decomposition and guarantees its existence.

We first formulate a vector-space generalization of the stable set problem on a bipartite graph.
Let $U$ and $V$ be finite dimensional vector spaces 
over a field ${\bf F}$.
Let $A: U \times V \to {\bf F}$ be a bilinear form.
Let ${\cal U}$ and ${\cal V}$ be the lattices of all vector subspaces of $U$ and of $V$, respectively. 
Let ${\cal L}$ and ${\cal M}$ be sublattices of ${\cal U}$ and of ${\cal V}$, respectively.
For subspace $(X,Y) \in {\cal L} \times {\cal M}$, 
let $A(X,Y)$ denote the image of $(X,Y)$ by $A$:
\[
A(X,Y) = \{A(x,y) \mid x \in X, y \in Y\}.
\]
Then either $A(X,Y) = \{0\}$ or ${\bf F}$.
A subspace $(X,Y) \in {\cal L} \times {\cal M}$
is said to be {\em stable} if 
$A(X,Y) = \{0\}$.
The {\em maximum stable subspace problem (MSSP)} on  
$({\cal L}, {\cal M}, A)$ is formulated as follows:
\begin{description}
	\item[MSSP:] Find a stable subspace $(X,Y) \in {\cal L} \times {\cal M}$ such that 
	$\dim X + \dim Y$ is maximum.
\end{description}
A stable subspace $(X,Y)$ is said to be {\em maximum} 
if $\dim X + \dim Y$ is maximum among all stable subspaces.
\begin{Lem}\label{lem:submo}
	Let $(X,Y)$ and $(X',Y')$ be stable subspaces.
	\begin{itemize}
		\item[{\rm (1)}]  
		Both $(X \cap X', Y + Y')$ and $(X + X', Y \cap Y')$ are stable.
		\item[{\rm (2)}] If both $(X,Y)$ and $(X',Y')$ are maximum, 
		then both $(X \cap X', Y + Y')$ and $(X + X', Y \cap Y')$ are maximum.
	\end{itemize}
\end{Lem}
\begin{proof}
	(1). By $A(X, Y \cap Y') = A(X', Y \cap Y') = \{0\}$,  
	we have $A(X + X', Y \cap Y') = A(X,Y \cap Y') + A(X', Y \cap Y') = \{0\}$. (2) follows from (1) and (\ref{eqn:dim}). 
\end{proof}
A canonical block-triangular matrix representation of $A$ 
is obtained from the family ${\cal S}_{\max} \subseteq {\cal L} \times {\cal M}$ of maximum stable subspaces.
Define a partial order $\preceq$ on ${\cal S}_{\max}$ 
by:  $(X,Y) \preceq (X',Y')$ if $X \subseteq X'$ (and $Y \supseteq Y'$).
Then ${\cal S}_{\max}$ is isomorphic to 
a sublattice of ${\cal L}$, and is a modular lattice.
Consider a maximal chain $(X^0,Y^0) \prec (X^1,Y^1) \prec \cdots \prec (X^h,Y^h)$ of ${\cal S}_{\max}$.
Let $i_k := \dim X^{k}$ and $j_k : = \dim Y^k$ for $k=0, 1,2,\ldots,h$.
Then $0 \leq i_0 < i_1 < i_2 < \cdots < i_h \leq n$, 
and $m \geq j_0 > j_1 > \cdots > j_h \geq 0$.
Since $i_k + j_k$ is the dimension of maximum stable subspaces,
it holds that
\[
i_{k+1} - i_{k} = j_{k} - j_{k+1} \quad (k=0,1,2,\ldots,h-1).
\]
Let $\{ e_1,e_2,\ldots,e_n\}$ be a basis of $U$ 
such that $\{e_1,e_2,\ldots,e_{i_k}\}$ is a basis of $X^{k}$ for $k=0,1,2,\ldots,h$. 
Also, let $\{f_1,f_2,\ldots,f_m\}$ be a basis of $V$
such that $\{f_1,f_2,\ldots, f_{j_k}\}$ is a basis of $Y^k$ for $k=0,1,2,\ldots,h$.

Let $a_{ij} := A(e_{n + 1 - i}, f_j)$ for $i=1,2,\ldots,n$ and $j=1,2,\ldots,m$. 
Consider $n \times m$ matrix $A_{\rm DM} = (a_{ij})$, which is the matrix representation of $A$ with respect to 
the basis $\{e_i\}$ and $\{f_j\}$.
For $k = 0,1,2,\ldots,h$, 
let $D_{k}$ denote the submatrix of $A_{\rm DM}$
consisting of $A(e_{i},f_{j})$ 
for $i_{k-1} < i \leq i_{k}$ and $j_{k}  < j  \leq j_{k-1}$, 
where $i_{-1} := 0$, $j_{-1} := m$.
Let $D_{\infty}$ denote the submatrix 
consisting of $A(e_{i},f_{j})$
for $i_{h} < i \leq n$ and $1 \leq j < j_h$.
Notice that $D_{0}$ and $D_{\infty}$ 
may be empty matrices.
Since $A(e_{i},f_{j}) = 0$ 
for $i \leq i_k, j \leq j_l$ with $k \leq l$,
the matrix representation $A_{\rm DM}$ of $A$ is in a block-triangular form as follows:
\begin{equation*}
A_{\rm DM} =
\left( 
\begin{array}{ccccc}
D_{ \infty} & & & & \\
 O & D_{h} & & & \\ 
 O & O  & \ddots & &  \\
 \vdots & \vdots & \ddots  & \ D_{1} & \\[0.7em]
 O      & O      & \cdots & O & D_{0}  
\end{array}
\right).
\end{equation*}
This matrix representation $A_{\rm DM}$ of bilinear form $A$ is called the {\em DM-representation with respect to} $({\cal L}, {\cal M})$.
Here the diagonal block $D_{k}$ is a square matrix 
of size $d_k := i_{k+1} - i_{k} = j_{k} - j_{k+1}$
for $k=1,2,\ldots,h$. 
In the case of $V \in {\cal M}$, 
$D_{0}$ is the empty matrix if and only if 
$(X^0,Y^0) = (0,V)$, i.e., $i_k+j_k = m$.
Consequently,
if $D_0$ is nonempty,  
then $m < i_0 + j_0$ and 
the column size of $D_{0}$ is less than the row size. 
Similarly, in the case of $U \in {\cal L}$, 
the row size of $D_{\infty}$ 
is less than its column size when $D_{\infty}$ is nonempty.

Although there are 
degrees of freedom in defining the entries of $A_{\rm DM}$,
its block structure is uniquely determined in the following sense.
Obviously the sizes of $D_{0}$ and $D_{\infty}$ are 
independent of the choice of a maximal chain.
Moreover the size sequence $(d_1,d_2,\ldots,d_h)$ is 
uniquely-determined up to permutation. 
This is a consequence of the Jordan-H\"older theorem.
See~\cite{ItoIwataMurota94} or \cite[Section 4.8]{MurotaBook} for 
more details on the uniqueness. 

It should be noted that 
Ito, Iwata, and Murota~\cite{ItoIwataMurota94}
focus on a block-triangular form, called a 
{\em proper block-triangular} form,
satisfying a more stronger rank condition: $\rank D_0 =  m - j_0$,  
$\rank D_{\infty} = n - i_h$, and
$\rank D_k = d_k$ for 
$k=1,2,\ldots,h$.
A proper block-triangular form does not exist in general.
If it exists, then it is obtained 
from a chain of maximum stable subspaces as above;  
see also Remark~\ref{rem:surplus}.

Let us return to the case of partitioned matrices.
Let $A = (A_{\alpha \beta})$ be 
an $n \times m$ partitioned matrix of the type $(n_1,n_2,\ldots,n_\mu; m_1,m_2,\ldots,m_\nu)$,  
where the entries of $A$ are numbers in a field ${\bf F}$. 
Let $U = {\bf F}^{n}$ 
and $V = {\bf F}^{m}$, where vectors in $U$ and $V$ 
are treated as column vectors.
Regard $A$ as a bilinear form on $U \times V $ by 
\[
(x,y) \mapsto x^{\top} A y.
\]

For $\alpha=1,2,\ldots,\mu$,
let $U_\alpha$ be the subspace of $U$ consisting of  
$u = (u_1,u_2,\ldots,u_n)^{\top}$ such that 
$u_i = 0$ unless $\sum_{\kappa=1}^{\alpha-1} n_\kappa < i \leq \sum_{\kappa=1}^{\alpha} n_\kappa$. 
Also, for $\beta =1,2,\ldots,\nu$,
let $V_\beta$ be the subspace of $V$ consisting of  
$v = (v_1,v_2,\ldots,v_m)^{\top}$ such that 
$v_j = 0$ unless $\sum_{\kappa=1}^{\beta-1} m_\kappa < j \leq \sum_{\kappa=1}^{\beta} m_\kappa$. 
Then $U = U_1 \oplus U_2 \oplus \cdots \oplus U_\mu$ 
and $V = V_1 \oplus V_2 \oplus \cdots \oplus V_\nu$.
Let ${\cal L}$ be the lattice of all subspaces $X$ of $U$ represented 
as $X = X_1 \oplus X_2 \oplus \cdots \oplus X_\mu$, 
where $X_\alpha$ is a subspace of $U_{\alpha}$ 
for $\alpha = 1,2,\ldots,\mu$.
Let ${\cal M}$ be the lattice of all subspaces $Y$ of $V$ represented 
as $Y = Y_1 \oplus Y_2 \oplus \cdots \oplus Y_\nu$,
where $Y_\beta$ 
is a subspace of $V_\beta$ for $\beta = 1,2,\ldots,\nu$.
For $X \in {\cal L}$ and $Y \in {\cal M}$, such representations are  unique, since $X_{\alpha} = X \cap U_{\alpha}$ and $Y_{\beta} = Y \cap V_{\beta}$.
Now $A_{\alpha \beta}$ is viewed as 
a bilinear form on $U_{\alpha} \times  V_{\beta}$.
Then the subspace $(X,Y) \in {\cal L} \times {\cal M}$ is stable if and only if 
\[
A_{\alpha \beta}(X_\alpha, Y_{\beta}) = \{0\}
\]
for $\alpha=1,2,\ldots,\mu$ and $\beta = 1,2,\ldots,\nu$.

Consider the DM-representation $A_{\rm DM}$ of $A$ with respect to 
$({\cal L}, {\cal M})$.
As above, choose a maximal chain
$(X^0,Y^0) \prec (X^1,Y^1) \prec \cdots \prec (X^h,Y^h)$ and bases $\{e_1,e_2,\ldots, e_n\}$ and $\{f_1,f_2,\ldots,f_m\}$, 
where these bases are chosen so that
$\{e_1,e_2,\ldots,e_n\} \cap U_{\alpha}$ is a basis of $U_{\alpha}$ for $\alpha=1,2,\ldots,\mu$, and
$\{f_1,f_2,\ldots,f_m\} \cap V_{\beta}$ 
is a basis of $V_{\beta}$ for $\beta = 1,2,\ldots,\nu$.

Let $E := (e_n \ e_{n-1} \ \cdots \ e_1)$ be the nonsingular $n \times n$ 
matrix of column vectors $e_i$, 
and let $F := (f_1 \ f_2 \ \cdots \ f_m)$ be the nonsingular $m \times m$ matrix of column vectors $f_j$. 
Then $A_{\rm DM}$ is given by $E^{\top} A F$, 
and is called the {\em DM-decomposition of $A$}.
Notice that $A \mapsto E^{\top} A F$ is 
an admissible transformation in the introduction.
Indeed, $E$ and $F$ are represented as
\[
E = 
 \left(
\begin{array}{cccc}
	E_1 & O  & \cdots & O  \\
	O & E_2 & \ddots & \vdots \\
	\vdots & \ddots & \ddots &   O\\
	O & \cdots & O& E_{\mu}
\end{array}
\right) P, \quad  F= 
\left(
\begin{array}{cccc}
	F_1 & O  & \cdots & O  \\
	O & F_2 & \ddots & \vdots \\
	\vdots & \ddots & \ddots &   O\\
	O & \cdots & O& F_{\nu}
\end{array}
\right)
Q,
\]
where $E_\alpha$ is a nonsingular $n_\alpha \times n_{\alpha}$ matrix, 
$F_{\beta}$ is a nonsingular $m_{\beta} \times m_{\beta}$ matrix, and
 $P$ and $Q$ are permutation matrices of sizes $n$ and $m$, respectively.
\begin{Rem}[Submodularity]
\label{rem:submo}
Let ${\cal S} \subseteq {\cal L} \times {\cal M}$ 
be the family of all stable subspaces.
By Lemma~\ref{lem:submo}~(1), 
${\cal S}$ is a modular lattice with 
join $(X,Y) \vee (X,Y') = (X+X',Y\cap Y')$ and 
meet $(X,Y) \wedge (X',Y') = (X \cap X', Y + Y')$. 
Then (MSSP) is viewed as a {\em modular function} maximization 
over ${\cal S}$. Indeed, by (\ref{eqn:dim}),  
the objective function $v(X,Y) := \dim X + \dim Y$ of (MSSP) 
satisfies the modular equality
\[
v(X,Y) + v(X',Y') = v((X,Y) \wedge (X',Y')) + v((X,Y) \vee (X',Y')) \quad ((X,Y),(X',Y') \in {\cal S}).
\]
Also (MSSP) can be formulated as 
a supermodular function maximization on the modular lattice ${\cal M}$.
For $Y \in {\cal M}$, 
the family of $X \in {\cal L}$ 
with $(X,Y) \in {\cal S}$
forms a sublattice of ${\cal L}$.
Let $Y^{\bot}$ denote
the maximum of this sublattice 
($\simeq$ the subspace in ${\cal L}$ orthogonal to $Y$ with respect to $A$).
Then (MSSP) is equivalent to: 
\begin{description}
	\item[MSSP$'$:] Find $Y \in {\cal M}$ 
	such that $\dim Y + \dim Y^\bot$ is maximum.
\end{description}
This is a supermodular function maximization over ${\cal M}$.
Indeed, the function $\gamma$ on ${\cal M}$ defined 
by $Y \mapsto \dim Y + \dim Y^{\bot}$ is {\em supermodular:}
\begin{equation*}
\gamma(Y) + \gamma (Y') \leq \gamma(Y + Y') + \gamma(Y \cap Y')
\quad (Y,Y' \in {\cal M}). 
\end{equation*}
This immediately follows from (\ref{eqn:dim}) and
$Y^{\bot} + Y'^{\bot} \subseteq (Y \cap Y')^{\bot}$ 
and $Y^{\bot} \cap Y'^{\bot} = (Y + Y')^{\bot}$.
\end{Rem}
\begin{Rem}[Relation to the original approach]\label{rem:surplus}
In the case of a partitioned matrix, 
the function $\gamma$ is equal to 
$n$ minus the {\em surplus function} by Ito, Iwata, and Murota~\cite{ItoIwataMurota94}, where
they considered a DM-type decomposition obtained from
a maximal chain of the minimizer set of the surplus function.
In their formulation, a partitioned matrix $A$ is viewed 
as a linear map $V \to U^* = U^*_1 \oplus U^*_2 \oplus \cdots \oplus U^*_\mu$. 
The surplus function $p$ on ${\cal M}$ is defined by
$Y \mapsto \sum_{\alpha} \dim {\rm Pr}_{\alpha} A(Y) - \dim Y$, where ${\rm Pr}_{\alpha}: U^{*} \to U^{*}_{\alpha}$ is the projection.
Observe that $(Y^\bot)_{\alpha}$ is 
the orthogonal space of $ {\rm Pr}_{\alpha} A(Y)$. 
Thus $\dim Y^{\bot} = \sum_{\alpha} \dim (Y^{\bot})_{\alpha} 
= \sum_{\alpha} n_\alpha - \dim {\rm Pr}_{\alpha} A(Y)$, 
and $\gamma = n - p$.
Consequently, our definition of the DM-decomposition is 
the same as the one in~\cite{ItoIwataMurota94}.
\end{Rem}
\begin{Rem}[VCSP perspective]	
We see in the next section that (MSSP) 
for our special case is solvable in polynomial time, 
though we do not know whether (MSSP) is tractable in general.
We here remark another nontrivial tractable case.
Let ${\cal U}_{\alpha}$ and ${\cal V}_{\beta}$ be the lattices of all vector subspaces of $U_{\alpha}$ and of $V_{\beta}$, respectively.
Then ${\cal L}$ is isomorphic to 
${\cal U}_1 \times {\cal U}_2 \times \cdots \times {\cal U}_{\mu}$,
and ${\cal M}$ is isomorphic to 
${\cal V}_1 \times {\cal V}_2 \times \cdots \times {\cal V}_{\nu}$.
Consider the case where  ${\bf F}$ is a finite field, and
$|{\cal U}_{\alpha}|$ and $|{\cal V}_{\beta}|$ are constants for all $\alpha, \beta$.
An example is: 
${\bf F} = {\rm GF}(2)$,
$n_\alpha = m_{\beta} = 2$, 
and $|{\cal U}_{\alpha}| = |{\cal V}_{\beta}| = 5$ for each $\alpha, \beta$.
In this case, (MSSP) is viewed as 
a {\em valued constraint satisfaction problem~(VCSP)}; see \cite{KTZ13}.
Define $S_{\alpha \beta}: {\cal U}_{\alpha} \times {\cal V}_{\beta}  \to \{0, - \infty \}$ by $S_{\alpha \beta}(X_{\alpha},Y_{\beta}) := 0$ if $A_{\alpha \beta}(X_{\alpha},Y_{\beta}) = \{0\}$ and $- \infty$ otherwise. 
Then (MSSP) can be formulated as
\begin{eqnarray*}
{\rm Max.}  && \sum_{\alpha} \dim X_{\alpha} + \sum_{\beta} \dim Y_{\beta} + \sum_{\alpha, \beta} S_{\alpha \beta} (X_{\beta}, Y_{\beta}) \\
{\rm s.t.} && (X_{1},X_{2},\ldots,X_{\mu}) \in {\cal U}_1 \times {\cal U}_2 \times \cdots \times {\cal U}_{\mu}, \\
&& (Y_{1},Y_{2},\ldots,Y_{\nu}) \in {\cal V}_1 \times {\cal V}_2 \times \cdots \times {\cal V}_\nu.
\end{eqnarray*}
This is a binary VCSP, 
where the input is the table of all function values of 
$\dim $ and $S_{\alpha \beta}$, and the total size is $O(nm)$.
By Remark~\ref{rem:submo}, 
this is a submodular VCSP, and hence
it admits a {\em fractional polymorphism} corresponding 
to the submodularity inequality.
Therefore, by the result of~\cite{KTZ13}, 
the {\em basic LP relaxation} is exact, and (MSSP) can be solved in polynomial time. Based on this, 
\cite{NakashimaHirai17} gave a polynomial time algorithm to compute the DM-decomposition in this setting.
\end{Rem}

\begin{Rem}[Rank of $A$]
Let $v^*$ be the optimal value of (MSSP).
The quantity $n + m - v^*$ 
is an upper bound of the rank of $A$; see~\cite{Lovasz89}.
Of course, this bound is far from being tight; 
the tight case is exactly the case 
where the obtained DM-decomposition
is a proper block-triangular form.
A natural situation for this bound to be effective is the case where
entries of distinct submatrices have no algebraic relations.
A {\em generic partitioned matrix}~\cite{IwataMurota95} 
is a notion to capture this situation, and
is a partitioned matrix $(A_{\alpha \beta})$ (over a field ${\bf F}$) 
such that
$A_{\alpha \beta}$ is represented as $t_{\alpha \beta} B_{\alpha \beta}$, 
where
$B_{\alpha \beta}$ is a matrix over a subfield ${\bf K}$ of ${\bf F}$
and elements $t_{\alpha \beta}$ 
are algebraically independent over ${\bf K}$.
For a multilayered mixed matrix $A$ 
(a generic partitioned matrix $A$ of type $(n_1,n_2,\ldots,n_{\mu};1,1,\cdots,1)$), 
the CCF theory~\cite{MurotaBook,MurotaIriNakamura87} 
implies that the bound $n+m-v^*$ 
is equal to $\rank A$.
This is also in the case for our rank-$1$ situation by~\cite[Theorem 1$^*$]{Lovasz89}.

Iwata and Murota~\cite{IwataMurota95} 
proved that this equality holds 
for any generic partitioned matrix 
with $n_{\alpha} \leq 2, m_{\beta} \leq 2$.
They also presented an example of
a $6 \times 6$ generic partitioned matrix of type 
$(3,3;2,2,2)$ not having this property.
This example consists of rank-$2$ submatrices. 
\end{Rem}
\section{DM-decomposition of a partitioned matrix\\ with rank-1 blocks}\label{sec:DM-rank1}

Let $A = (A_{\alpha \beta})$ be an $n \times m$ partitioned matrix of type $(n_1,n_2,\ldots,n_\mu; m_1,m_2,\ldots,m_{\nu})$. 
Suppose now that 
\begin{description}
	\item[(rank-1 condition)] $\rank A_{\alpha \beta} \leq 1$ for $\alpha =1,2,\ldots,\mu$, $\beta = 1,2,\ldots,\nu$.
\end{description}
In the following,  
index $\alpha$ ranges over $1,2,\ldots,\mu$ and 
$\beta$ ranges over $1,2,\ldots,\nu$.

\subsection{Reduction}
Here we show that (MSSP) under the rank-1 condition is reduced 
to an independent matching problem.
The reduction is based on
an intuitively obvious fact (Lemmas~\ref{lem:rank1} and \ref{lem:cover}) 
that for any stable subspace $(X,Y)$,  
the component $X_{\alpha}$ (resp. $Y_{\beta}$) belongs to 
the intersection of 
$\ker A^{\top}_{\alpha \beta}$ (resp. $\ker A_{\alpha \beta}$)
for several $\beta$ (resp. ${\alpha}$).

Let $\varPi_\alpha$ be the set of 
hyperplanes $\pi$ in $U_\alpha$
such that $\pi = \ker A^{\top}_{\alpha \beta}$ 
for some $\beta$ (with $\rank A_{\alpha \beta} = 1$).
Also, let $\varSigma_\beta$ be the set of 
hyperplanes $\sigma$ in $V_\beta$
such that $\sigma = \ker A_{\alpha \beta}$ for some $\alpha$.
Let $\varPi$ be the disjoint union of $\varPi_\alpha$ over $\alpha$, and
let $\varSigma$ be the disjoint union of $\varSigma_\beta$ over $\beta$. 
Define bipartite graph $G = (\varPi,\varSigma, E)$  
on vertex set $\varPi \cup \varSigma$, where
for each $\alpha,\beta$,
\begin{itemize}
	\item[] $(\pi,\sigma) \in \varPi_{\alpha} \times \varSigma_{\beta}$
	is an edge in $E$ if and only if 
	$\pi = \ker A_{\alpha \beta}^{\top}$ and $\sigma = \ker A_{\alpha \beta}$. 
\end{itemize}
An example is given soon (Example~\ref{ex:1}).

For $H \subseteq \varPi$, define subspace $X(H) = X(H)_1 \oplus X(H)_2 \oplus \cdots \oplus X(H)_{\mu}$ of ${\cal L}$ by
$X(H)_\alpha :=$ 
the intersection of all hyperplanes in $H \cap \varPi_\alpha$.
Similarly,
for $K \subseteq \varSigma$, 
define subspace $Y(K) = Y(K)_{1} \oplus Y(K)_{2} \oplus \cdots \oplus Y(K)_{\nu}$ of ${\cal M}$ by 
$Y(K)_\beta :=$ the intersection of
all hyperplanes in $K \cap \varSigma_\beta$.

Consider matroids ${\bf M}(\varPi_\alpha)$ and ${\bf M}(\varSigma_\beta)$
defined in Section~\ref{subsec:matroid}.
Let ${\bf M}(\varPi)$ be 
the direct sum of matroids ${\bf M}(\varPi_\alpha)$ over $\alpha$, and
let ${\bf M}(\varSigma)$ be 
the direct sum of matroids ${\bf M}(\varSigma_\beta)$ over $\beta$.
Consider the independent matching problem on $({\bf M}(\varPi), {\bf M}(\varSigma), G)$.
Then the maximum stable subspace is obtained 
from a minimum cover in $({\bf M}(\varPi), {\bf M}(\varSigma), G)$ as follows.
\begin{Thm}\label{thm:min-max}
	Let $A$ be an $n \times m$ partitioned matrix 
	satisfying the rank-1 condition.
	\begin{itemize}
		\item[{\rm (1)}] A stable subspace $(X,Y)$ is maximum if and only if
		$(X,Y)$ is represented as $(X(H),Y(K))$ for a minimum cover $(H,K)$
		in $({\bf M}(\varPi), {\bf M}(\varSigma), G)$.  
		\item[{\rm (2)}] The maximum dimension of a stable subspace is 
		equal to $n+m$ 
		minus the maximum size of an independent matching in $({\bf M}(\varPi), {\bf M}(\varSigma), G)$. 
	\end{itemize}
\end{Thm}
\begin{Ex}\label{ex:1}
	Consider the following partitioned matrix of type $(2,2,2;2,2,2)$ 
	over field ${\bf F} = {\rm GF}(2)$:
	\[
	A = 
	\left(
	\begin{array}{cc|cc|cc}
	1 & 0 & 1 & 1 & 0 & 0 \\
	0 & 0 & 1 & 1 & 1 & 1 \\
	\hline
	1 & 1 & 1 & 1 & 1 & 0 \\
	0 & 0 & 0 & 0 & 1 & 0 \\
	\hline
	1 & 0 & 1 & 1 & 1 & 0 \\
	1 & 0 & 1 & 1 & 0 & 0
	\end{array}
	\right)
	\]
	This matrix satisfies the rank-1 condition.
	There are three hyperplanes in $U_{\alpha}$ and in $V_{\beta}$. 
	These hyperplanes are identified with their normal vectors
	$(1\ 0)$, $(0\ 1)$, and $(1\ 1)$.     
	Then an edge between $\pi \in \varPi_{\alpha}$ and 
	$\sigma \in \varSigma_{\beta}$ is given if and only if $A_{\alpha \beta} = \pi^{\top} \sigma$, such as 
	$A_{13} = (0\ 1)^{\top} (1\ 1)$ and $A_{22} = (1\ 1)^\top (1\ 1)$.
	The graph $G$ is given in Figure~\ref{fig:G}.
		\begin{figure}[t]
			\begin{center}
				\includegraphics[scale=0.7]{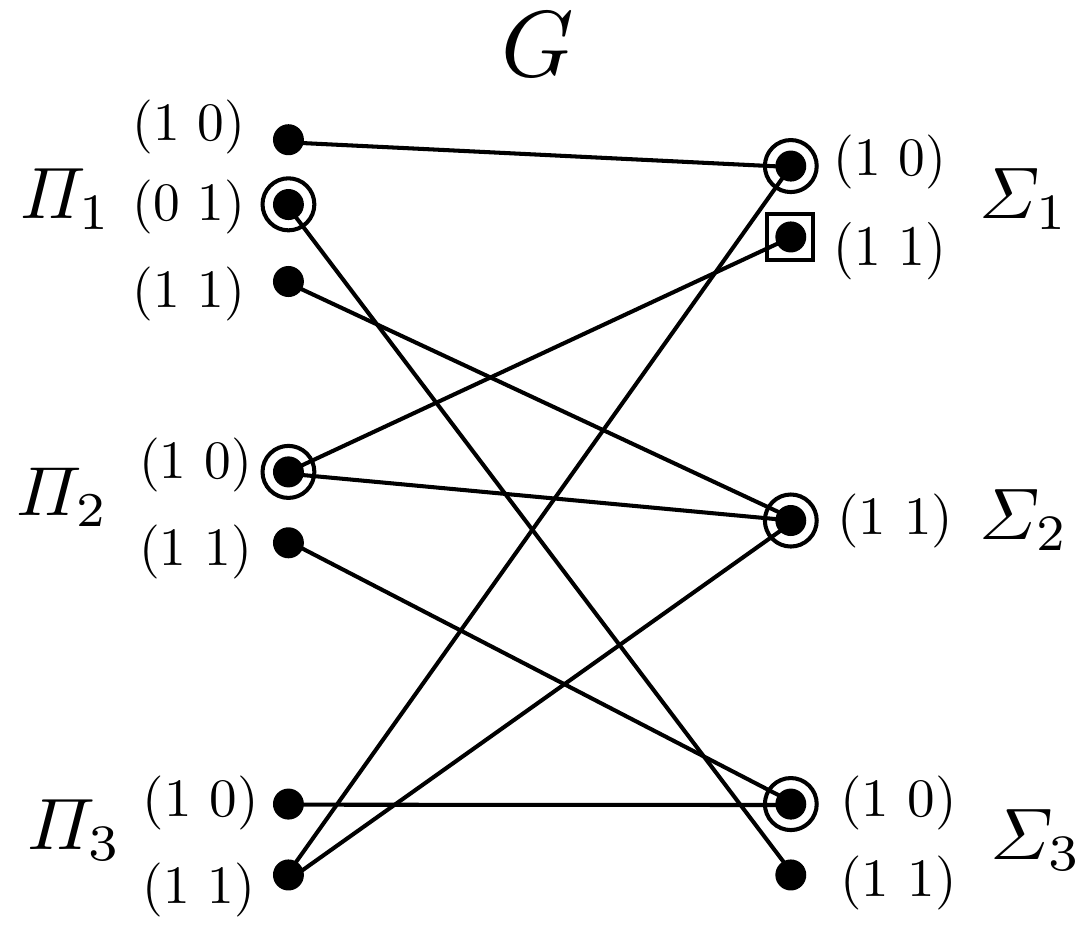}
				\caption{Bipartite graph $G$}
				\label{fig:G}
			\end{center}
		\end{figure}\noindent
	Now matroid ${\bf M}(\varPi_{\alpha})$ (resp. ${\bf M}(\varSigma_{\beta})$) is 
	the matroid of linear dependence of normal vectors
	(of hyperplanes) in $\varPi_{\alpha}$ (resp. $\varSigma_{\beta}$).
	Consider circled vertices in this figure, 
	which forms a cover $(H,K)$. 
	The corresponding subspace $(X(H),Y(K))$ is given by
	\begin{eqnarray*}
	X(H) &= &{\bf F} (1\ 0)^{\top} \oplus {\bf F} (0\ 1)^{\top} \oplus {\bf F}^2, \\
	Y(K) & = & {\bf F} (0\ 1)^{\top} \oplus {\bf F} (1\ 1)^{\top} \oplus {\bf F} (0\ 1)^{\top},
	\end{eqnarray*}
	where its dimension (obviously) matches 
	the matroid quantity $6 + 6 - \rho^+(H) - \rho^-(K) = 7$.
	One can verify that $(X(H), Y(K))$ is indeed stable.
	For example, 
	\[
	A_{11}(X(H)_1,Y(K)_1) = {\bf F} 
		\left(
		\begin{array}{cc}
		1  & 0  
		\end{array}
		\right)
	 \left(
	\begin{array}{cc}
	1 & 0 \\
	0 & 0
	\end{array}
	\right) 
	\left(
	\begin{array}{c}
	0  \\
	1  
	\end{array}
	\right)  = \{ 0\}.
	\]
	On the other hand, if $(0\ 1) \in \varSigma_{1} \cap K$ is replaced 
	by squared vertex $(1\ 1)$, then $(H, K)$ is not a cover, and $(X(H),Y(K))$ 
	is not stable:
	 \[
	 A_{11}(X(H)_1,Y(K)_1) = {\bf F} 
	 		\left(
	 		\begin{array}{cc}
	 		1  & 0  
	 		\end{array}
	 		\right)
	  \left(
	 \begin{array}{cc}
	 1 & 0 \\
	 0 & 0
	 \end{array}
	 \right) 
	 \left(
	 \begin{array}{c}
	 1  \\
	 1  
	 \end{array}
	 \right)  = {\bf F}.
	 \] 
\end{Ex}

The rest of this subsection is devoted to proving this theorem.
Let $A = (A_{\alpha \beta})$ be a partitioned matrix satisfying the rank-1 condition.
The following property of a rank-$1$ matrix is essential in our reduction.
\begin{Lem}\label{lem:rank1}
	For subspaces $X_{\alpha} \subseteq U_{\alpha}$ and $Y_{\beta} 
	\subseteq V_{\beta}$, 
	if $A_{\alpha \beta}(X_{\alpha}, Y_{\beta}) = \{0\}$, then $X_{\alpha} \subseteq \ker A_{\alpha \beta}^{\top}$ or $Y_{\beta} \subseteq \ker A_{\alpha \beta}$.
\end{Lem}
\begin{proof}
	Suppose to the contrary that $A_{\alpha \beta}(X_{\alpha}, Y_{\beta}) = \{0\}$,  $X_{\alpha} \not \subseteq \ker A_{\alpha \beta}^{\top}$, and $Y_{\beta} \not \subseteq \ker A_{\alpha \beta}$.
	   Then $0 \neq X_{\alpha} \neq U_{\alpha}$, $0 \neq Y_{\beta} \neq V_{\beta}$, 
		$A_{\alpha \beta}(X_{\alpha},V_{\beta}) \neq \{0\}$, and $A_{\alpha \beta}(U_{\alpha},Y_{\beta}) \neq \{0\}$. 
		There are nonzero vectors 
		$u \in X_{\alpha}, u' \in U_{\alpha} \setminus X_{\alpha}$, $v \in Y_{\beta}, v' \in V_{\beta} \setminus Y_{\beta}$
		such that $A_{\alpha \beta}(u,v') \neq 0$, $A_{\alpha \beta}(v,u') \neq 0$, and $A_{\alpha \beta}(u,v) = 0$.
		Thus
		\[
		\det \left(
		\begin{array}{ccc}
		A_{\alpha \beta}(u,v') & A_{\alpha \beta}(u',v') \\
		A_{\alpha \beta}(u,v) & A_{\alpha \beta}(u',v)
		\end{array}
		\right) = 
		\det \left(
		\begin{array}{ccc}
		A_{\alpha \beta}(u,v') & A_{\alpha \beta}(u',v') \\
		0 & A_{\alpha \beta}(u',v)
		\end{array}
		\right)
		\neq 0.
		\]
		This means that a $2 \times 2$ submatrix of a matrix obtained from $A_{\alpha \beta}$ by change of basis 
		has nonzero determinant, and hence $\rank A_{\alpha \beta} \geq 2$. 
		This contradicts $\rank A_{\alpha \beta} \leq 1$.
\end{proof}
For a subspace $(X,Y)$ in ${\cal L} \times {\cal M}$, 
let $\varPi_{\alpha,X}$ denote the set of hyperplanes  in $\varPi_\alpha$ containing $X_\alpha$, and
let $\varSigma_{\beta,Y}$ denote the set of  
hyperplanes in $\varSigma_{\beta}$ containing $Y_\beta$.
Let $\varPi_X$ be the union of $\varPi_{\alpha,X}$ over $\alpha$, 
and let $\varSigma_Y$ be the union of $\varPi_{\beta,Y}$ over $\beta$. 
\begin{Lem}\label{lem:cover}
	\begin{itemize}
\item[{\rm (1)}] If $(H,K) \in \varPi \times \varSigma$ is a cover, then $(X(H), Y(K))$ is a stable subspace.
\item[{\rm (2)}] If $(X,Y)$ is a stable subspace, 
	then $(\varPi_X, \varSigma_Y)$ is a cover.
	\item[{\rm (3)}] If $(X,Y)$ is a maximum stable subspace, 
	then
	$(X,Y) = (X({\varPi_X}), Y({\varSigma_Y}))$.
	\end{itemize}
\end{Lem}
\begin{proof}
	(1).
	Suppose that $(X,Y) = (X(H),Y(K))$ is not stable.
	Then $A_{\alpha \beta}(X_\alpha,Y_\beta) \neq \{0\}$ for some $\alpha,\beta$.
	This implies that $X_\alpha \not \subseteq \ker A_{\alpha \beta}^{\top}$ and 
	$Y_\beta \not \subseteq \ker A_{\alpha \beta}$.
	Therefore $\varPi_X$ does not contain $\ker A_{\alpha \beta}^{\top}$ 
	and $\varSigma_Y$ does not contain $\ker A_{\alpha \beta}$.
	By $H \subseteq \varPi_X$ and $K \subseteq \varSigma_Y$, 
	the edge between $\ker A_{\alpha \beta}^{\top}$ and $\ker A_{\alpha \beta}$
	does not meet $H \cup K$. Namely $(H,K)$ is not a cover.
	
	(2). Suppose that $(X, Y)$ is stable.
	Consider arbitrary $\alpha,\beta$.
	Now $A_{\alpha \beta}(X_\alpha,Y_\beta) = \{0\}$.
	By Lemma~\ref{lem:rank1}, 
	it holds that $X_\alpha \subseteq \ker A_{\alpha \beta}^{\top}$ 
	or $Y_\beta \subseteq \ker A_{\alpha \beta}$.
	 This means that the endpoints of edge joining 
	 $\ker A_{\alpha \beta}^\top$ and $\ker A_{\alpha \beta}$
	 meet $\varPi_{\alpha,X} \cup \varSigma_{\beta,Y}$.

	(3). Suppose that $(X,Y)$ is a maximum stable subspace.
	By (1) and (2),  subspace $(X({\varPi_X}), Y({\varSigma_Y}))$ is also stable.
	By $X \subseteq X({\varPi_X})$, $Y \subseteq Y({\varSigma_Y})$, and the maximality,
	it must hold that $(X,Y) = (X({\varPi_X}), Y({\varSigma_Y}))$.
\end{proof}
Thus the problem (MSSP) under the rank-1 condition is equivalent to: 
\begin{description}
	\item[MC:] Find a cover $(H,K)$ 
	such that $\dim X(H) + \dim Y(K)$ is maximum.
\end{description}

%
Now $\dim X(H)$ is equal to $n$ minus the rank of $H$ in ${\bf M}(\varPi)$, and  $\dim Y(K)$ is equal to $m$ minus the rank of $K$ in
${\bf M}(\varSigma)$. Namely, (MC) is 
nothing but the minimum cover problem dual to the independent matching problem on $({\bf M}(\varPi), {\bf M}(\varSigma), G)$.
This proves Theorem~\ref{thm:min-max}.

\subsection{Algorithm}
Here we present an algorithm to compute the DM-decomposition $A_{\rm DM}$.
Let $M$ be a maximum independent matching, 
which is obtained by the algorithm in Section~\ref{subsec:matroid} 
with $V^+ = \varPi$, $V^- = \varSigma$, 
${\bf M}^+ = {\bf M}(\varPi)$, and ${\bf M}^-= {\bf M}(\varSigma)$.
From $\tilde G_{M}$, 
we are going to construct a compact representation of ${\cal S}_{\max}$.
Let $C_0$ be the set of vertices $v$ having a path from $S$ to $v$, 
and let $C_{\infty}$ the set of vertices $v$ having a path from $v$ to $T$.
Let $H_0, H_{\infty}, K_0, K_{\infty}$ be the subsets of vertices defined by
\begin{eqnarray*}
&& H_0 := C_{0} \cap \partial^+ M,\quad  K_0 := C_{0} \cap \partial^- M, \\
&& H_\infty := C_{\infty} \cap \partial^+ M,\quad K_\infty := C_{\infty} \cap \partial^- M.
\end{eqnarray*}
Let
$\tilde G_{M}'$ be the digraph obtained from $\tilde G_{M}$ by deleting $C_0$ and $C_{\infty}$.
Consider the strongly connected component decomposition of $\tilde G'_M$. Let $h$ be the number of components 
meeting $\partial^+ M \cup \partial^- M$. 
Consider a partition 
$\{H_0, H_1,H_2,\ldots,H_h,H_{\infty}\}$ of $\partial^+ M$
such that $\pi$ and $\pi'$ belong to $H_{k}$ $(1 \leq k \leq h)$ if and only 
if $\pi$ and $\pi'$ belong to the same component.
Accordingly, consider a partition 
$\{K_0, K_1,K_2,\ldots,K_h,K_{\infty}\}$ of $\partial^- M$
such that $H_k$ is matched to $K_k$ by $M$ for $k=1,2,\ldots,h$ 
(or $K_k$ belongs to the same component as $H_k$).
Define a partial order $\preceq$ on ${\cal P} := \{1,2,\ldots,h\}$ such that
$k \preceq l$ if and only if  there is a directed path in $\tilde G_M'$ from $H_l$ to $H_k$.
For an ideal $J \in {\cal J}({\cal P})$, define $H_J \subseteq \varPi$ and 
$K_J \subseteq \varSigma$ by
\[
H_J := \bigcup \{ H_{k} \, \mid \, k \in {\cal P} \cup \{\infty\} \setminus J\}, 
\quad
K_J := \bigcup \{ K_{k} \, \mid \, k \in J \cup \{0\}\}.
\]

\begin{Prop}
	${\cal J}({\cal P})$ is isomorphic to ${\cal S}_{\max}$, 
	where an isomorphism is given by
	\begin{equation}
	{\cal J}({\cal P}) \ni J \mapsto 
	(X(H_J), 
	Y(K_J)). 
	\end{equation}
In particular, 	${\cal S}_{\max}$ is isomorphic to a distributive sublattice of ${\cal L}$.
\end{Prop}  
\begin{proof}
	Let $J \in {\cal J}({\cal P})$.
	Let $C'$ be the set of vertices in $\tilde G_M$ reachable 
	from a vertex in $\bigcup_{k\in J} H_{k}$.
	Let $C:= C' \cup C_{0}$.
	Then $S \subseteq C$, $T \cap C = \emptyset$, and
	no edge goes out from $C$.
	Hence $(\varPi \setminus C, \varSigma \cap C)$ is a minimum cover (Theorem~\ref{thm:min-max}~(2)).
	Thus $(X(\varPi \setminus C),Y(\varSigma \cap C))$ is a maximum stable subspace (Theorem~\ref{thm:min-max}~(2)).
	By definition of $C$ and $J \in {\cal J}({\cal P})$,
	it holds $\partial^+ M \setminus C = H_J$ and 
	$\partial^- M \cap C = K_J$.
	Also the rank of $\varPi \setminus C$ in ${\bf M}(\varPi)$ 
	is equal to $|\partial^+ M \setminus C|$, and  the rank of $\varSigma \cap C$ in ${\bf M}(\varSigma)$ 
	is equal to $|\partial^- M \cap C|$ (Lemma~\ref{lem:indep}~(3))
	Thus $X(H_J) = X(\varPi \setminus C)$, $Y(K_J) = Y(\varSigma \cap C)$, and $(X(H_J),Y(K_J)) \in {\cal S}_{\max}$.

	Conversely, let $(X,Y)$ be a maximum stable subspace. 
	Then $(\varPi_X, \varSigma_Y)$ is a minimum cover. 
	By Lemma~\ref{lem:indep}, 
	$(\varPi_X, \varSigma_Y) = (\varPi \setminus C, \varSigma \cap C)$ holds
	for some $C$ such that $S \subseteq C$, $T \cap C = \emptyset$, 
	and ($*$) no edge goes out from $C$.
	Let $J$ be the set of indices $k \in \{1,2,\ldots,h\}$ such that $H_k$ belongs to $C$.
	Then, by the property ($*$), $J$ is an ideal of ${\cal P}$.
\end{proof}

Thus, from the strongly connected component decomposition of $\tilde G'_{M}$, 
we obtain a poset ${\cal P}$ representing ${\cal S}_{\max}$ as ${\cal S}_{\max} \simeq {\cal J}({\cal P})$. 
Relabel ${\cal P} = \{1,2,\ldots,h\}$ so that $k \prec l$ implies $k < l$ for $k,l \in {\cal P}$.  
Then $\{1,2,\ldots,k\}$ is an ideal.
For $k=0,1,2\ldots,h$, 
define stable subspace $(X^{k},Y^k)$ by
\[
X^k := X(H_{k+1} \cup H_{k+2} \cup \cdots \cup H_{h} \cup H_{\infty}), \quad
Y^k := Y(K_0 \cup K_1 \cup K_2 \cup \cdots \cup K_k).	
\]
Then $(X^0,Y^0) \prec (X^1,Y^1) \prec \cdots \prec (X^h,Y^h)$ is a maximal chain in ${\cal S}_{\max}$.

Next we construct bases of $U$ and of $V$ 
to obtain $A_{\rm DM}$.
A hyperplane in $\varPi_{\alpha}$
is identified with its normal vector, 
which is an $n_{\alpha}$-dimensional row vector. 
Similarly a hyperplane in $\varSigma_{\beta}$
is identified with an $m_{\beta}$-dimensional row vector.
Each submatrix $A_{\alpha \beta}$ of rank $1$ is represented 
as $c \pi^{\top} \sigma$ for some
$\pi \in \varPi_{\alpha}$, $\sigma \in \varSigma_{\beta}$, 
and $c \in {\bf F} \setminus \{0\}$.
Now $\varPi_\alpha$ is a set of 
$n_\alpha$-dimensional row vectors,
and $\varSigma_\beta$ is a set of $m_\beta$-dimensional 
row vectors.
${\bf M}(\varPi_\alpha)$ and ${\bf M}(\varSigma_\beta)$ are 
matroids of linear independence of these vectors.
For each $\alpha$, choose any set $\tilde \varPi_{\alpha}$ of vectors such that 
$(\partial^+ M \cap \varPi_{\alpha}) \cup \tilde \varPi_{\alpha}$ 
is a basis of (the dual space of) ${\bf F}^{n_{\alpha}}$, and
add  $\tilde \varPi_{\alpha}$ to $H_0$.
Let $H := H_{0} \cup H_1 \cup \cdots \cup H_h \cup H_{\infty}$.
Similarly, for each $\beta$, 
choose any set $\tilde \varSigma_{\beta}$ of vectors such that
$(\partial^- M \cap \varSigma_{\beta}) \cup \tilde \varSigma_{\beta}$ 
is a basis of (the dual space of) ${\bf F}^{m_{\beta}}$, and $\tilde \varSigma_{\beta}$ to $K_{\infty}$.
Let $K := K_{0} \cup K_1 \cup \cdots \cup K_h \cup K_{\infty}$.
Suppose that $H = \{\pi_{1}, \pi_2,\ldots,\pi_n\}$, 
where indices are ordered as: 
if $\pi_i \in H_{k}$, $\pi_{j} \in H_{l}$, and $k < l$, 
then $i < j$.
Suppose that $K = \{\sigma_1,\sigma_2,\ldots,\sigma_m\}$, where
indices are ordered as: 
if $\sigma_i \in K_{k}$, $\sigma_{j} \in K_{l}$, and $k < l$, 
then $i < j$.

Then 
$\varPi_\alpha \cap H  = \{\pi_{\alpha_1},\pi_{\alpha_2}, \ldots,  \pi_{\alpha_{n_\alpha}}\}$ for $\alpha_{1} < \alpha_{2} < \cdots < \alpha_{n_\alpha}$ and
$\varSigma_\beta \cap K = \{\pi_{\beta_1},\pi_{\beta_2}, \ldots,  \pi_{\beta_{m_\beta}}  \}$ for $\beta_1 < \beta_2 < \cdots < \beta_{m_\beta}$.
Let $R_\alpha$ be the nonsingular $n_\alpha \times n_\alpha$ matrix 
whose $\lambda$th row vector is $\pi_{\alpha_\lambda}$, 
and let $S_\beta$ be the nonsingular $m_\beta \times m_\beta$ matrix whose $\lambda$th row vector is $\sigma_{\beta_\lambda}$: 
\begin{equation}
R_\alpha := \left(
\begin{array}{ccccc}
  &  & \pi_{\alpha_{1}} & & \\
  &  & \pi_{\alpha_{2}} &  & \\
  &  &     \vdots     &  & \\
   &  &  \pi_{\alpha_{n_\alpha}}   &   & 
\end{array} \right) \quad 
S_\beta := \left(
\begin{array}{ccccc}
&  & \sigma_{\beta_{1}} & & \\
&  & \sigma_{\beta_{2}} &  & \\
&  &     \vdots     &  & \\
&  &  \sigma_{\beta_{m_\beta}}   &   & 
\end{array} \right).
\end{equation}
Let $E_\alpha$ be a nonsingular $n_\alpha \times n_\alpha$ matrix
such that $R_\alpha E_\alpha$ is upper-triangular.
Let $e_{\alpha,\lambda}$ denote the $\lambda$th column vector of $E_{\alpha}$.
Let $F_\beta$ be a nonsingular $m_\beta \times m_\beta$ matrix
such that $S_\beta F_\beta$ is lower-triangular.
Let $f_{\beta,\lambda}$ denote the $\lambda$th column vector of $F_{\beta}$.
For $i=1,2,\ldots,n$, define an $n$-dimensional vector $e_i$ as follows.
Suppose that $\alpha_\lambda = i$.
For $j$ with $\sum_{k=0}^{\alpha-1} |H_k| < j \leq \sum_{k=0}^{\alpha} |H_k|$, 
the $j$th component of $e_i$ is equal to 
the $(j - \sum_{k=0}^{\alpha-1} |H_k|)$th component of 
$e_{\alpha,\lambda}$. 
All  other components of $e_i$ are defined to be zero. 
For $j=1,2,\ldots,m$, define $m$-dimensional vector $f_j$ as follows.
Suppose that $\beta_\lambda = j$.
For $i$ with $\sum_{k=0}^{\beta-1} |K_k| < i \leq \sum_{k=0}^{\beta} |K_k|$, 
the $i$th component of $f_j$ is equal to 
the $(i - \sum_{k=0}^{\beta-1} |K_k|)$th component of 
$f_{\beta,\lambda}$. 
All other components of $f_j$ are defined to be zero. 

Then the DM-decomposition $A_{\rm DM}$ of $A$ is given by
\[
A_{\rm DM} = (e_{n}\ e_{n-1}\ \cdots \ e_1)^{\top} A (f_{m}\ f_{m-1}\ \cdots \ f_1),
\]
which is verified by the next lemma.
\begin{Lem}
	For $k=0,1,2,\ldots,h$, the following hold:
	\begin{itemize}
		\item[{\rm (1)}] $\{e_1,e_2,\ldots,e_{i_k}\}$ is a basis of $X^k$ with $i_k = \sum_{l=0}^k |H_l|$. 
		\item[{\rm (2)}] $\{f_{j_{k}+1},f_{j_{k}+2},\ldots, f_{m}\}$ is a basis of $Y^k$ with $j_k = \sum_{l=0}^k |K_l|$.
	\end{itemize}
\end{Lem}
\begin{proof}
	(1) Suppose that $\varPi_\alpha \cap H = \{\pi_{\alpha_1},\pi_{\alpha_2}, \ldots,  \pi_{\alpha_{n_\alpha}}  \}$ for $\alpha_{1} < \alpha_{2} < \cdots < \alpha_{n_\alpha}$. 
	Then $\varPi_\alpha \cap (H_{k+1} \cup H_{k+2} \cup \cdots \cup H_{\infty}) = 
	\{\pi_{\alpha_{\lambda}},\pi_{\alpha_{\lambda+1}}, \ldots,  \pi_{\alpha_{n_\alpha}} \}$ for the minimum $\lambda$ with $\sum_{l=0}^{k}|H_l| <  \alpha_\lambda$.
	Then $(X^k)_{\alpha} := X(H_{k+1} \cup H_{k+2} \cup \cdots \cup H_{\infty} 
	)_{\alpha}$ is the intersection of hyperplanes $\pi_{\alpha_\lambda},\pi_{\alpha_{\lambda+1}}, \ldots,  \pi_{\alpha_{n_\alpha}}$.
    Since $R_\alpha E_\alpha$ is upper-triangular,  
    all $e_{\alpha_1},e_{\alpha_2},\ldots,e_{\alpha_{\lambda-1}}$ belong to $(X^k)_{\alpha}$, and span $(X^k)_{\alpha}$.
    Consequently, the statement holds.
    
    (2) Suppose that $\varSigma_\beta \cap K = \{\sigma_{\beta_1},\sigma_{\beta_2}, \ldots,  \sigma_{\beta_{m_\beta}}  \}$ for $\beta_1 < \beta_2 < \cdots < \beta_{m_\beta}$.
    Then $\varSigma_\beta \cap (K_{0} \cup K_{1} \cup \cdots \cup K_{k}) = 
    \{\sigma_{\beta_{1}},\sigma_{\beta_{2}}, \ldots,  \sigma_{\beta_\lambda}\}$ for the maximum $\lambda$ with  $\beta_{\lambda} \leq \sum_{l=0}^{k}|K_l|$.
	Then $(Y^k)_\beta := Y(K_{0} \cup K_{1} \cup \cdots \cup K_{k})_{\beta}$ is the intersection of hyperplanes $\sigma_{\beta_{1}},\sigma_{\beta_{2}}, \ldots,  \sigma_{\beta_\lambda}$.
	Since $S_\beta F_\beta$ is lower-triangular, all $f_{\beta_{\lambda+1}},f_{\beta_{\lambda+2}},\ldots,f_{\beta_{m_\beta}}$ belong to $(Y^k)_{\beta}$, and span $(Y^k)_{\beta}$.
\end{proof}

Finally we give a rough estimate of the time complexity 
to compute $A_{\rm DM}$.
Suppose that each submatrix $A_{\alpha \beta}$ of rank $1$ is 
given as an expression $\pi^{\top} \sigma$ for row vectors $\pi, \sigma$.
The bipartite graph $G = (\varPi,\varSigma, E)$ has $O(\mu \nu)$ vertices and $O(\mu \nu)$ edges.
Therefore the number of iterations of the independent matching algorithm 
is bounded by $\mu \nu$.
In the construction of $\tilde{G}_{M}$,
the edges added to $\varPi_{\alpha}$ are identified by
Gaussian elimination in $O(n_{\alpha}^2 |\varPi_{\alpha}|)$ time.
Consequently we can construct $\tilde{G}_{M}$ 
in $O(n^2 \nu + m^2 \mu)$ time.
An augmenting path is found in $O(\mu \nu)$ time.
Thus a maximum independent matching $M$ is 
obtained in $O(\mu^2 \nu^3 n^2  +\mu^3 \nu^2 m^2) = O(n^4 m^3+ n^3m^4)$ time.
The poset ${\cal P}$ is naturally obtained 
from the final $\tilde{G}'_M$ by the strongly connected component decomposition, and 
a maximal chain is also naturally obtained.
Matrices $E_{\alpha}, F_{\beta}$ 
are obtained in $O(n^3 + m^3)$ time.
The matrix
$(e_{n}\ e_{n-1}\ \cdots \ e_1)^{\top} A (f_{m}\ f_{m-1}\ \cdots \ f_1)$ is calculated in $O(nm^2 + n^2m)$ time.
The total time is $O(n^4 m^3 + n^3 m^4)$. This proves Theorem~\ref{thm:main}.
\begin{Ex}
Consider the partitioned matrix $A$ in Example~\ref{ex:1}.
According to the above algorithm, 
the DM-decomposition $A_{\rm DM}$ is computed as follows.
%
Three hyperplanes (of normal vectors) $(1\ 0)$, $(0\ 1)$, 
and $(1\ 1)$ in
$U_{\alpha}$ are denoted by $\alpha a$, $\alpha b$, and $\alpha c$, respectively.
Similarly, 
three hyperplanes $(1\ 0)$, $(0\ 1)$, 
and $(1\ 1)$ in
$V_{\beta}$ are denoted by $\beta' a$, $\beta' b$, and $\beta'c$,
respectively.
Then $\varPi_{1} = \{1a, 1b,1c\}$, $\varPi_{2} = \{2a,2c\}$, 
$\varPi_{3} = \{3a, 3c\}$, $\varSigma_1 = \{1'a, 1'c\}$, $\varSigma_2 =\{2'c\}$, and $\varSigma_3 = \{3'a,3'c\}$. 
Consider the independent matching problem on 
$({\bf M}(\varPi), {\bf M}(\varSigma), G)$.
	\begin{figure}[t]
		\begin{center}
			\includegraphics[scale=0.7]{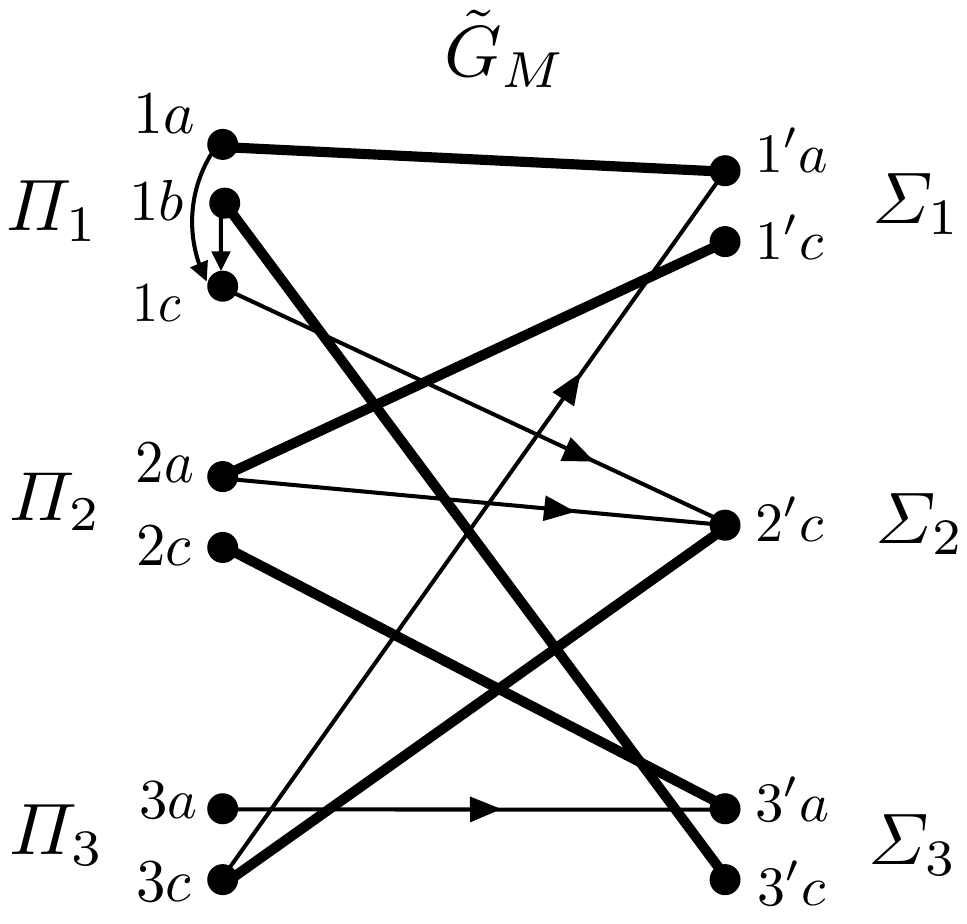}
			\caption{Auxiliary digraph $\tilde G_M$ for a maximum independent matching $M$}
			\label{fig:GM}
		\end{center}
	\end{figure}\noindent
Figure~\ref{fig:GM} 
depicts the auxiliary digraph $\tilde G_{M}$ 
for a maximum independent matching 
$M =\{(1a,1'a), (1b,3'c), (2a,1'c), (2c,3'a), (3c,2'c)\}$, where two directed edges corresponding to an edge in $M$ 
is drawn by a single thick undirected edge.
Now $S = \{3a\}$, $T = \emptyset$,  $C_0 = \{3a,3'a, 2c\}$, and 
$C_{\infty} = \emptyset$.
There are three strongly connected components 
in $\tilde G_M'$ meeting $M$.
Then
$H_k$ and $K_k$ $(k=0,1,2,3,\infty)$ are given as
\[
\begin{array}{ll}
H_{\infty} = \emptyset,  & K_{\infty} = \emptyset, \\
H_3 = \{1b \},   & K_{3} = \{3'c\}, \\
H_2 = \{2a \},   & K_{2} = \{1'c\}, \\
H_1 = \{1a,3c \},   & K_{1} = \{1'a, 2'c\}, \\
H_0 = \{2c\},   & K_{0} = \{3'a\},
\end{array}
\]
where the partial order $\preceq$ on ${\cal P} = \{1,2,3\}$ given by
\[
2 \succ 1 \prec 3.
\]
Add $3a$ to $H_0$.
The elements in $H = \bigcup_{k} H_{k}$
are ordered
as $2c, 3a, 1a,3c, 2a,1b$.
Add $2'a$ to $K_{\infty}$.  
The elements in 
$K = \bigcup_{k} K_{k}$ are ordered
as $3'a,1'a,2'c,1'c,3'c,2'a$.
Matrices $R_1,R_2, R_3, S_1, S_2, S_3$ are given by
\begin{eqnarray*}
R_1 = \left( \begin{array}{c} 1a \\ 1b  \end{array}\right) 
= \left( \begin{array}{cc} 1 & 0\\ 0 & 1  \end{array}\right), && 
S_1 =  \left( \begin{array}{c} 1'a \\ 1'c  \end{array}\right) 
= \left( \begin{array}{cc} 1 & 0\\ 1 & 1 \end{array} \right),\\
R_2 = \left( \begin{array}{c} 2c \\ 2a  \end{array}\right) 
= \left( \begin{array}{cc} 1 & 1\\ 1 & 0  \end{array}\right), && 
S_2 = \left( \begin{array}{c} 2'c \\ 2'a  \end{array}\right) 
= \left( \begin{array}{cc} 1 & 1\\ 0 & 1  \end{array}\right),\\
R_3 = \left( \begin{array}{c} 3a \\ 3c  \end{array}\right) 
= \left( \begin{array}{cc} 1 & 0\\ 1 & 1  \end{array}\right), && 
S_3 = \left( \begin{array}{c} 3'a \\ 3'c  \end{array}\right) =
\left( \begin{array}{cc} 1 & 0\\ 1 & 1  \end{array}\right).		
\end{eqnarray*}
Matrices $E_1,E_2,E_3,F_1,F_2,F_3$ 
(for which $R_{\alpha} E_\alpha$ is upper-triangular and $S_{\beta} F_{\beta}$ is lower-triangular) are given by
\begin{eqnarray*}
	E_1 = \left( \begin{array}{cc} 1 & 0\\ 0 & 1  \end{array}\right), && F_1 = \left( \begin{array}{cc} 1 & 0\\ 0 & 1 \end{array} \right),\\
	E_2 = \left( \begin{array}{cc} 0 & 1\\ 1 & 0  \end{array}\right), && F_2 = \left( \begin{array}{cc} 1 & 1\\ 0 & 1  \end{array}\right),\\
	E_3 = \left( \begin{array}{cc} 1 & 0\\ 1 & 1  \end{array}\right), && F_3 = \left( \begin{array}{cc} 1 & 0\\ 0 & 1  \end{array}\right).		
\end{eqnarray*}
Then the DM-decomposition $A_{\rm DM}$ is given by
\begin{eqnarray*}
\setlength{\dashlinedash}{0.5pt}
\setlength{\dashlinegap}{0.5pt}
A_{\rm DM} & =& 
\left(
\begin{array}{cccccc}
	0 &   &   & 1 &   &  \\
	1 &   &   & 0 &   &  \\
	  & 1 &   &   &   & 0 \\
	  & 0 &   &   &   & 1 \\
	  &   & 0 &   & 1 &  \\
	  &   & 1 &   & 1 & 
\end{array}
\right)^{\top}
\left(
\begin{array}{cc|cc|cc}
	1 & 0 & 1 & 1 & 0 & 0 \\
	0 & 0 & 1 & 1 & 1 & 1 \\
	\hline
	1 & 1 & 1 & 1 & 1 & 0 \\
	0 & 0 & 0 & 0 & 1 & 0 \\
	\hline
	1 & 0 & 1 & 1 & 1 & 0 \\
	1 & 0 & 1 & 1 & 0 & 0
\end{array}
\right) 
\left(
\begin{array}{cccccc}
	  &   & 0 &   & 1 &   \\
	  &   & 1 &   & 0 &   \\
	1 &   &   & 1 &   &   \\
	1 &   &   & 0 &   &   \\
	  & 0 &   &   &   & 1 \\
	  & 1 &   &   &   & 0
\end{array}
\right)\\
& = & \left(
\setlength{\dashlinedash}{0.5pt}
\setlength{\dashlinegap}{0.5pt}
\begin{array}{cccccc}
\cdashline{1-2}
  & \multicolumn{1}{:c:}{1} & 0 & 1 & 0 & 1 \\
\cdashline{2-3}
  &    & \multicolumn{1}{:c:}{1} & 1 & 1 & 1 \\
\cdashline{3-5}
  &    &          & \multicolumn{1}{:c}{1}&  \multicolumn{1}{c:}{1} &0 \\
  &   &          & \multicolumn{1}{:c}{1}&  \multicolumn{1}{c:}{1} &0 \\
\cdashline{4-6}
  &   &    &   &   & \multicolumn{1}{:c:}{1} \\
  &   &    &   &   & \multicolumn{1}{:c:}{1} \\
\cdashline{6-6} 
\end{array}
\right),
\end{eqnarray*}
where all empty entries are zero, 
and diagonal blocks $D_{0},D_1,D_2,D_3, D_{\infty} (=\emptyset)$ 
are indicated by dashed boxes. 
\end{Ex}

\section*{Acknowledgments}
We thank Kazuo Murota and 
Shin-ichi Tanigawa for comments and discussion. 
We also thank Yuni Iwamasa for careful reading, and the referee for helpful comments.
The work was partially supported by JSPS KAKENHI Grant Numbers JP25280004, JP26330023, JP26280004, JP17K00029.

\end{document}